\documentclass[12pt,a4paper,reqno]{amsart}

\usepackage{a4wide}
\usepackage{amssymb,amsmath,amsthm,mathrsfs}
\usepackage{graphicx}
\usepackage{float}
\usepackage{colortbl}
\usepackage{enumerate}

\theoremstyle{plain}
\newtheorem{theorem}{Theorem}
\newtheorem{lemma}{Lemma}[section]

\theoremstyle{definition}
\newtheorem{definition}{Definition}

\newtheorem*{condition}{Condition}

\theoremstyle{remark}
\newtheorem{remark}{Remark}

\def\R{\ensuremath{\mathbb R}}
\def\N{\ensuremath{\mathbb N}}

\def\Z{\ensuremath{\mathbb Z}}

\def\e{{\ensuremath{\rm e}}}

\def\l{{\rm Leb}}

\def\p{\ensuremath{\mathbb P}}

\def\QQ{\ensuremath{\mathscr Q}}

\def\ie{{\em i.e.}, }

\def\cv{\ensuremath{\text {Cor}}}

\def\o{\ensuremath{\underline{\omega}}}
\def\x{\ensuremath{\underline{\xi}}}

\def\dist{\ensuremath{\text{dist}}}

\def\eps{\varepsilon}

\def\spp{\text{SP\negmedspace}_{p,
\theta}}

\def\sm{\ensuremath{\mu_\varepsilon}}

\numberwithin{equation}{section}

\begin{document}

\title{Extreme value statistics for dynamical systems with noise}

\author[D. Faranda]{Davide Faranda}
\address{Davide Faranda\\ Klimacampus, Institute of Meteorology, University of Hamburg\\
Grindelberg 5, 20144\\ Hamburg \\Germany
}
\email{davide.faranda@zmaw.de}
\urladdr{http://www.mi.uni-hamburg.de/Davide-Faran.6901.0.html?\&L=3}

\author[J.M. Freitas]{Jorge Milhazes Freitas}
\address{Jorge Milhazes Freitas\\ Centro de Matem\'atica \& Faculdade de Ci\^encias da Universidade do Porto\\
Rua do Campo Alegre 687\\ 4169-007 Porto\\ Portugal
}
\email{jmfreita@fc.up.pt}
\urladdr{http://www.fc.up.pt/pessoas/jmfreita}

\author[V. Lucarini]{Valerio Lucarini}
\address{Valerio Lucarini\\ Klimacampus, Institute of Meteorology, University of Hamburg\\
Grindelberg 5, 20144\\ Hamburg \\Germany. -- Department of Mathematics and Statistics, University of Reading\\ Reading \\ UK
}
\email{valerio.lucarini@zmaw.de}
\urladdr{http://www.mi.uni-hamburg.de/index.php?id=6870\&L=3}

\author[G. Turchetti]{Giorgio Turchetti}
\address{Giorgio Turchetti\\ Department of Physics, University of Bologna.INFN-Bologna\\
Via Irnerio 46\\ Bologna \\Italy }
\email{turchetti@bo.infn.it}
\urladdr{http://www.unibo.it/SitoWebDocente/default.htm?UPN=giorgio.turchetti40unibo.it}

\author{ Sandro Vaienti}
\address{Sandro Vaienti\\ 
UMR-7332 Centre de Physique Th\'{e}orique, CNRS, Universit\'{e}
d'Aix-Marseille I, II, Universit\'{e} du Sud, Toulon-Var and FRUMAM,
F\'{e}d\'{e}ration de Recherche des Unit\'{e}s des Math\'{e}matiques de Marseille,
CPT Luminy, Case 907, F-13288 Marseille CEDEX 9}
\email {vaienti@cpt.univ-mrs.fr}

\thanks{{\em Acknowledgements} \ SSV was supported by the ANR-Project {\em Perturbations}, by the  CNRS-PEPS {\em Mathematical Methods of Climate Theory} and by the {\sc PICS} ( Projet International de Coop\'eration Scientifique), {\em Propri\'et\'es statistiques des syst\`emes dynamiques det\'erministes et al\'eatoires}, with the University of Houston,  n. PICS05968. Part of this work was done while he was visiting the {\em Centro de Modelamiento Matem\'{a}tico, UMI2807}, in Santiago de Chile with a CNRS support (d\'el\'egation). SV thanks H. Aytach for useful discussions. DF acknowledges Jeroen Wouters for useful discussions. DF and VL acknowledge the support of the NAMASTE project. NAMASTE has received funding from the European Research Council under the European Community's Seventh Framework Programme (FP7/2007-2013) / ERC Grant agreement No. 257106. JMF is thankful to A.C.M. Freitas and M. Todd for fruitful conversations. JMF was partially supported by FCT (Portugal) grant SFRH/BPD/66040/2009, by FCT project PTDC/MAT/099493/2008 and CMUP, which is financed by FCT through the programs POCTI and POSI, with national
and European Community structural funds. JMF and SV are also supported by FCT project PTDC/MAT/120346/2010, which is financed by national and European Community structural funds through the programs  FEDER and COMPETE. }


\keywords{Random dynamical systems, Extreme Values, Hitting Times Statistics, Extremal Index} \subjclass[2010]{
37A50, 60G70, 37B20, 60G10, 37A25, 37H99}

\begin{abstract}
We study the distribution of maxima ({\em Extreme Value Statistics}) for sequences of observables computed along orbits generated by random transformations. The underlying, deterministic, dynamical system can be regular or chaotic. In the former case, we will show that by perturbing rational or irrational rotations with additive noise, an extreme value law  appears, regardless of the intensity of the noise, while unperturbed rotations do not admit such limiting distributions. In the case of deterministic chaotic dynamics, we will consider observables specially designed to study the recurrence properties in the neighbourhood of periodic points. Hence, the exponential limiting law for the distribution of maxima is  modified by the presence of the {\em extremal index}, a positive parameter not larger than one, whose inverse gives the average size of the clusters of extreme events. The theory predicts that such a parameter is unitary when the system is perturbed randomly. We perform sophisticated numerical tests to assess how strong is the impact of noise level, when finite time series are considered. We find agreement with the asymptotic theoretical results but also non-trivial behaviour in the finite range. In particular our results suggest that in many applications where finite datasets can be produced or analysed one must be careful in assuming that the smoothing nature of noise prevails over the underlying deterministic dynamics.
\end{abstract}

\maketitle

\section{Introduction}

The main purpose of this paper is to study the extremal behaviour of randomly perturbed dynamical systems. By extremal behaviour we mean its statistical performance regarding the existence of Extreme Value Laws (EVLs), or in other words, the existence of distributional limits for the partial maxima of stochastic processes arising from such systems. In many aspects of natural and social sciences and engineering, the statistical properties of the extremes of a system are usually relevant tied with actual risk assessment. This is an element of why the theory of extremes has received such a great deal of attention over the years. The motivation for considering randomly perturbed systems follows from the fact that, very often, dynamical systems are used to model human activities or natural phenomena and the fact that errors made by observations usually take a random character which can be well described by the random perturbation formalism. On the other hand, random noise is often added in numerical modelling in order to represent the lack of knowledge on (or practical impossibility of representing) some of the processes taking place in the system of interest, often characterised by small spatial and/or temporal scales, whose explicit representation is virtually impossible. See \cite{palmer2010stochastic} for a comprehensive discussion on this issue in a geophysical setting. Also, noise is often considered as a ``good'' component to add when performing numerical simulations, because of its ability to smoothen the invariant measure and basically remove unphysical solutions \cite{lucarini2012stochastic}.

 The distributional properties of the maxima of stationary sequences is driven by the appearance of exceedances of high thresholds. For independent and identically distributed (i.i.d.) processes, the exceedances appear scattered through the time line. For dependent processes this may not necessarily hold and clustering of exceedances may occur. The Extremal Index (EI), which we will denote by $0\leq\theta\leq 1$, is a parameter that quantifies the amount of clustering. No clustering means that $\theta=1$ and strong clustering means that $\theta$ is close to 0. For deterministic hyperbolic systems, it has been shown, in \cite{FFT12,FP12,K12,AFV12}, that a dichotomy holds:  the EI is always 1 at every points except at the (repelling) periodic points where the EI $0<\theta<1$. We remark that Hirata \cite{H93} had already observed the two types of behaviour but not the dichotomy. In \cite{AFV12}, the authors proved the first results (up to our knowledge) regarding the existence of EVLs for piecewise expanding systems which are randomly perturbed by additive noise. They observed that adding noise has a ``smoothing'' effect in terms of eliminating all possible clustering of exceedances. In \cite{AFV12}, the authors show that, for the randomly perturbed systems considered, the EI is always 1. The main tool used there was decay of correlations against $L^1$ observables, which held for the unperturbed and perturbed systems considered. We remark that this property implies an excellent mixing behaviour of both systems.

One of the main achievements of this paper is the extension of the results in \cite{AFV12}, by showing that adding noise still has a smoothing effect when the unperturbed system is not mixing at all, or even worse: when the system is actually periodic. To be more precise, for unperturbed systems, we consider rotations on the circle (irrational or not) and show that by adding absolutely continuous noise with respect to the Haar measure, we can always prove the existence of EVLs with the EI being 1 everywhere. The analysis of these systems will be made using Fourier transforms to compute decay of correlations of specific observables. We observe that for the original unperturbed systems the lack of mixing does not allow to prove EVLs in the usual sense. Hence, here, we have a more drastic transition which motivates the question of whether it is possible to distinguish numerically the real nature of the underlying (unperturbed) systems when we look at the extremal statistics of the randomly perturbed data.

The analytical discussion is supported with numerical simulations devised to show how, for finite samples, the extremal behaviour may or not follow the asymptotic results. This has obvious relevance in terms of applications where the amount of affordable  statistics is limited by the available technology.  One of the key aspects surfacing from the numerical analysis of extremes that we carried out is that the estimation of some dynamical and geometrical properties of the underlying physical measure strongly depends on a combination of the intensity of the noise and the length of the data sample. There are two main experimental issues deriving from stochastic perturbations. The first concerns with the general claim that the addition of noise has a ``smoothing'' effect over the physical measure and enhances the chaoticity of the underlying deterministic dynamical systems. As we have mentioned, we expect to see this in terms of extremal behaviour, at least asymptotically. This claim will be weighed up against the simulations showing that we may need to integrate the system for a very long time to observe any changes in the measure structure. Moreover, for higher dimensional systems featuring stable and unstable manifolds, the stochastic perturbations should be directed on the stable direction so that the ``smoothing'' effect of the measure is more effective. The second issue is related to the level of noise needed to modify the deterministic properties and, in this case in particular, the sensitivity of the extremal type behaviour to the the noise level. To discuss this latter issue we will use the fact that the numerical round off is comparable to a random noise on the last precision digit \cite{knuth1973art}. This observation will allow us to claim that, for systems featuring periodic or quasi periodic motions (as the rotations on the circle), such a noise level is not sufficient for producing detectable changes regarding the observed extremal behaviour. On the other hand, the analysis for generic points of chaotic systems, such as the full shift on three symbols, will show that the intrinsic chaotic behaviour of the map dims the effect of the noise and makes the perturbed system indistinguishable from a deterministic one. If, instead, the analysis is carried out at periodic points, the effect of the perturbation will be easily noticeable in the disruption of periodicity and its consequent clustering of exceedances, which can be detected by estimation of the EI. The disruption of periodicity will be further analysed through simulations on the quadratic map. We will show, for finite sample behaviour, how the EI depends on the intensity of the noise.  We will conclude the analysis suggesting a way to analyse the disruption of periodicity in general systems where  the nature of the noise may be unknown and by analysing the behaviour of the Pomeau-Manneville map introduced in  \cite{PM80} as example of intermittent chaos where the issues of having chaotic and regular behaviour coexists in certain parameter ranges.

We will also present some numerical simulations on a Lorenz map and the H\'enon map The former features a complicated structure of the bifurcation diagram which allow for analysing the role of noise addition when a system undergoes bifurcations. The latter  features a stable and an unstable directions, homoclinic tangencies and the coexistence of two different attractors. In a previous paper we have shown that it is possible to relate the EVL parameters to the local dimension of the attractor of a system possessing a singular invariant physical measure \cite{lucarini2012extreme}. This latter experiments will be the gateway to go beyond the maps presented in this paper and to suggest a general procedure to analyse issues related to the stochastic perturbations of dynamical systems in order to frame the analysis carried out here  in a more general setup. In fact, we find that it is far from trivial to detect the smoothing effect of noise on the invariant measure when finite time series are considered.

We would like to remark that such considerations have great relevance in the context of Axiom A systems and having statistical mechanical applications in mind when addressing the problem of the applicability of the fluctuation-dissipation theorem (FDT). While in the deterministic case the non-smooth nature of the measure along the stable manifold makes it impossible to apply straightforwardly the FDT (see discussion in \cite{ruelle1997differentiation, ruelle1998general, ruelle2009review, lucarini2008response, lucarini2011a}), the addition of even a very small amount of random forcing in principle ``cures'' the singularities of the measure and makes sure that in principle the FDT can be used to related forced and free motions (see \cite{abramov2007blended, lacorata2007fluctuation}). Our results suggest that one must be careful in assuming that such a desirable effect is really detectable when finite datasets are analysed.

\section{Extreme values, the system and the perturbation}

Our main purpose is to study the extremal behaviour of randomly perturbed dynamical systems. This will be done by analysing the partial maxima of stationary stochastic processes arising from such systems. We will start by presenting the main concepts regarding the extreme values on a general framework of stationary stochastic processes. Then we will turn to the construction of stationary stochastic processes deriving from dynamical systems and their random perturbations by additive noise, which will be the object of our study.  

\subsection{Extreme values -- definitions and concepts} 

Let $X_0, X_1, \ldots$ be a stationary stochastic process.
We denote by $F$ the cumulative distribution function (d.f.) of $X_0$, \ie $F(x)=\p(X_0\leq x)$. Given any d.f. $F$, let $\bar{F}=1-F$ and $u_F$ denote the right endpoint of the d.f. $F$, \ie
$
u_F=\sup\{x: F(x)<1\}.
$ We say we have an exceedance of the threshold $u<u_F$ whenever
$U(u):=\{X_0>u\}$
occurs. We define the sequence of partial maxima  $M_1, M_2,\ldots$ given by
$M_n=\max\{X_0,\ldots,X_{n-1}\}.$

\begin{definition}
We say that we have an \emph{Extreme Value Law} (EVL) for $M_n$ if there is a non-degenerate d.f. $H:\R\to[0,1]$ with $H(0)=0$ and,  for every $\tau>0$, there exists a sequence of levels $u_n=u_n(\tau)$, $n=1,2,\ldots$,  such that
\begin{equation}
\label{eq:un}
  n\p(X_0>u_n)\to \tau,\;\mbox{ as $n\to\infty$,}
\end{equation}
and for which the following holds:
$\p(M_n\leq u_n)\to \bar H(\tau)$, as $n\to\infty$,
where the convergence is meant in the continuity points of $H(\tau)$.
\end{definition}

The motivation for using a normalising sequence $u_n$ satisfying \eqref{eq:un} comes from the case when $X_0, X_1,\ldots$ are independent and identically distributed (i.i.d.). In this i.i.d. setting, it is clear that $\p(M_n\leq u)= (F(u))^n$, where $F$ is the d.f. of $X_0$, \ie $F(x):=\p(X_0\leq x)$. Hence, condition \eqref{eq:un} implies that
\begin{equation}
\label{eq:iid-maxima}
\p(M_n\leq u_n)= (1-\p(X_0>u_n))^n\sim\left(1-\frac\tau n\right)^n\to\e^{-\tau},
\end{equation}
as $n\to\infty$. Moreover, the reciprocal is also true (see \cite[Theorem~1.5.1]{LLR83} for more details). Note that in this case $H(\tau)=1-\e^{-\tau}$ is the standard exponential d.f.

When $X_0, X_1, X_2,\ldots$ are not independent, exceedances of high thresholds may have a tendency to appear in clusters, which creates the appearance of a parameter $\theta$ in the exponential law, called the Extremal Index: 
\begin{definition}
We say that $X_0,X_1,\ldots$ has an \emph{Extremal Index} (EI) $0\leq\theta\leq1$ if we have an EVL for $M_n$ with $\bar H(\tau)=\e^{-\theta\tau}$ for all $\tau>0$.
\end{definition}

We can say that, the EI is that measures the strength of clustering of exceedances in the sense that, most of the times,  
it can be interpreted as the inverse of the average size of the clusters of exceedances. In particular, if $\theta=1$ then the exceedances appear scattered along the time line without creating clusters.
%

In order to prove the existence of an EVL corresponding to an EI equal to 1, for general stationary stochastic processes, Leadbetter  \cite{L73} introduced conditions $D(u_n)$ and $D'(u_n)$, which are some sort of mixing and anti clustering conditions, respectively. However, since the first of these conditions is too strong to be checked when the stochastic processes are dynamically generated,  
motivated by Collet's work \cite{C01}, in \cite{FF08a} the authors suggested a condition $D_2(u_n)$ which together with $D'(u_n)$ was enough to prove the existence of an exponential EVL for maxima. In fact, \cite[Theorem~1]{FF08a} states that if the following conditions  hold for $X_0, X_1,\ldots$ then there exists an EVL for $M_n$ and $H(\tau)=1-e^{-\tau}$. 
\begin{condition}[$D_2(u_n)$]\label{cond:D2} We say that $D_2(u_n)$ holds for the sequence $X_0,X_1,\ldots$ if for all $\ell,t$
and $n$,
$|\p\left(X_0>u_n\cap
  \max\{X_{t},\ldots,X_{t+\ell-1}\leq u_n\}\right)-\p(X_0>u_n)
  \p(M_{\ell}\leq u_n)|\leq \gamma(n,t),$
where $\gamma(n,t)$ is decreasing in $t$ for each $n$ and
$n\gamma(n,t_n)\to0$ when $n\rightarrow\infty$ for some sequence
$t_n=o(n)$.
\end{condition}
Now, let $(k_n)_{n\in\N}$ be a sequence of integers such that
\begin{equation}
\label{eq:kn-sequence-1}
k_n\to\infty\quad \mbox{and}\quad  k_n t_n = o(n).
\end{equation}
\begin{condition}[$D'(u_n)$]\label{cond:D'} We say that $D'(u_n)$
holds for the sequence $X_0$, $X_1$, $X_2$, $\ldots$ if there exists a sequence $\{k_n\}_{n\in\N}$ satisfying \eqref{eq:kn-sequence-1} and such that
$\lim_{n\rightarrow\infty}\,n\sum_{j=1}^{\lfloor n/k_n \rfloor}\p( X_0>u_n,X_j>u_n)=0.$
\end{condition}

Condition $D_2(u_n)$ is  much weaker than the original $D(u_n)$, and it is easy to show that it follows easily from sufficiently fast decay of correlations (see \cite[Section~2]{FF08a}).

When $D'(u_n)$ does not hold, clustering of exceedances is responsible for the appearance of a parameter $0<\theta<1$ in the EVL which now is written as $\bar H(\tau)=\e^{-\theta \tau}$. 

In \cite{FFT12},  the authors established a connection between the existence of an EI less than 1 and periodic behaviour. This was later generalised for rare events point processes in \cite{FFT13}. 
Under the presence of periodic phenomena, the inherent rapid recurrence creates clusters of exceedances which makes it easy to check that condition $D'(u_n)$ fails (see \cite[Section~2.1]{FFT12}). 
To overcome this difficulty, in \cite{FFT12}, the authors considered the annulus 
\begin{equation}
\label{eq:def-Qp}
Q_p(u):=\{X_0>u, \;X_p\leq u\}
\end{equation} resulting from removing from $U(u)$ the points that were doomed to return after $p$ steps, which form the smaller ball $U(u)\cap f^{-p}(U(u))$. 
Then, the main crucial observation in \cite{FFT12} is that  the limit law corresponding to no entrances up to time $n$ into the ball $U(u_n)$ was equal to the limit law corresponding to no entrances into the annulus $Q_p(u_n)$ up to time $n$ (see \cite[Proposition~1]{FFT12}). This meant that, roughly speaking, the role played by the balls $U(u)$ could be replaced by that of the annuli $Q_p(u)$, with the advantage that points in $Q_p(u)$ were no longer destined to return after just $p$ steps. 

Based in this last observation, in \cite{FFT12}, the authors adapted conditions $D_2(u_n)$ and $D'(u_n)$ in order to obtain an EVL at repelling periodic points. 
In fact, the next two conditions 
can be described as being obtained from $D_2(u_n)$ and $D'(u_n)$ by replacing balls by annuli.  Let $\mathscr Q_i(u_n):=\bigcap_{j=0}^{i-1} f^{-j}(Q_p(u_n)^c)$. Note that while the occurrence of the event $\{M_n\leq u_n\}$ means that no entrance in the ball $\{X_0>u_n\}$ has occurred up to time $n$, the occurrence of $\mathscr Q_n(u_n)$ means that no entrance in the annulus $Q_p(u_n)$ has occurred up to time $n$.    
\begin{condition}[$D^p(u_n)$]\label{cond:Dp}We say that $D^p(u_n)$ holds
for the sequence $X_0,X_1,X_2,\ldots$ if for any integers $\ell,t$
and $n$
$\left|\p\left(Q_{p}(u_n)\cap
  f^{-t}(\QQ_{\ell}(u_n)\right)-\p(Q_{p}(u_n))
  \p(\QQ_{\ell}(u_n))\right|\leq \gamma(n,t),
$
where $\gamma(n,t)$ is non increasing in $t$ for each $n$ and
$n\gamma(n,t_n)\to0$ as $n\rightarrow\infty$ for some sequence
$t_n=o(n)$.  
\end{condition}
Again the above condition can easily be checked from sufficiently fast decay of correlations (see \cite[Section~3.3]{FFT12}).

\begin{condition}[$D'_p(u_n)$]\label{cond:D'p} We say that $D'_p(u_n)$
holds for the sequence $X_0,X_1,X_2,\ldots$ if there exists a sequence $\{k_n\}_{n\in\N}$ satisfying \eqref{eq:kn-sequence-1} (with $t_n$ given by condition $D^p(u_n)$) such that
$\lim_{n\rightarrow\infty}\,n\sum_{j=1}^{[n/k_n]}\p( Q_{p}(u_n)\cap
f^{-j}(Q_{p}(u_n)))=0.$
\end{condition}

In \cite[Theorem~1]{FFT12}, it was proved that if a stationary stochastic process satisfies conditions $\spp(u_n)$, $D^p(u_n)$ and $D'_p(u_n)$ then we have an EVL for $M_n$ with $\bar H(\tau)=\e^{-\theta\tau}$.

\subsection{Stochastic processes arising from randomly perturbed systems}
To simplify the exposition we will consider one dimensional maps $f$ on the circle $\mathbb{S}^1$ or on the unit interval $I$ (from now on we will use $I$ to identify both spaces),  provided, in the latter case, that the image $f(I)$ is strictly included into $I$. We will perturb them with additive noise, namely, we introduce the family of maps $$
f_{\omega}(x)=f(x)+\omega
$$
where we have to take the mod-$1$ operation if the map is considered on the circle. The quantity $\omega$ is chosen on the interval $\Omega_{\eps}=[-\eps,\eps]$ with distribution $\vartheta_{\eps}$, which we will take equivalent to Lebesgue on $\Omega_{\eps}$. Consider now an i.i.d.  sequence ${\omega_k}, \ k\in \mathbb{N}$ taking values on the interval $\Omega_{\eps}$ and distributed according to $\vartheta_{\eps}$. We construct the random orbit (or {\em random transformation}) by the concatenation

\begin{equation}
\label{eq:concatenated}
f_{\o}^n(x)=f_{\omega_n}\circ f_{\omega_{n-1}}\circ\cdots\circ f_{\omega_1}(x).
\end{equation}
 The role of the invariant measure is now played by the stationary measure $\mu_{\eps}$ which is defined as
\[
\iint \phi(f_{\o}(x)) \,d\mu_\varepsilon(x)\,d\vartheta_\varepsilon^\N(\o)=\int \phi(x)\,d\mu_\varepsilon(x),
\]
for every $\phi\in L^{\infty}$ ($L^{\infty}$ to be intended with respect to the
Lebesgue measure \l)\footnote{This choice is dictated by the fact that the stationary measure will be equivalent to Lebesgue in all the examples considered below}.
The previous equality could also be written as
$
\int {\mathcal U}_{\eps}\phi \ d\mu_{\eps}=\int \phi \ d\mu_{\eps}
$
where the operator ${\mathcal U}_{\eps}: L^{\infty}\rightarrow
L^{\infty}$, is defined as $({\mathcal U}_{\eps}\phi)(x)=\int
_{\Omega_{\eps}}\phi(f_{\omega}(x))d\vartheta _{\eps}$ and it is called
the {\em random evolution operator}. \\The decay of correlations of the perturbed systems could be formulated in terms of the random evolution operator. More precisely, we will take two non-zero observables, $\phi$ and $\psi$ and we will suppose that $\phi$ is of bounded variation with norm $||\cdot||_{BV}$ and $\psi\in L^1_m$\footnote{Of course we could do other choices, but this two spaces will play a major role in the subsequent theory.} Then the correlation integral is
$$
\cv_\eps(\phi,\psi, n):=\frac{1}{\|\phi\|_{BV}\|\psi\|_{L^1_m}} \left|\int {\mathcal U}_{\eps}^n\phi~\psi\, d\sm-\int  \phi\,
d\sm\int \psi\, d\sm\right|
$$
In the following we will use the product measure $\p :=\mu_{\eps} \times \vartheta_{\eps}^{\mathbb{N}}$.\\

We are now in condition of  defining the time series $X_0$, $X_1$, $X_2$,\dots arising from our system simply by evaluating a given observable  $\varphi:I\to\R\cup\{+\infty\}$ along the random orbits of the system:
\begin{equation}
\label{eq:def-rand-stat-stoch-proc-RDS2} X_n=\varphi\circ f^n_{\o},\quad \mbox{for
each } n\in {\mathbb N},
\end{equation}
We assume that $\varphi$ achieves a global maximum at $z\in I$; for every $u<\varphi(z)$ but sufficiently close to $\varphi(z)$, the event $\{y\in I:\; \varphi(y)>u\}=\{X_0>u\}$ corresponds to a topological ball ``centred'' at $z$ and, for every sequence $(u_n)_{n\in\N}$ such that $u_n\to \varphi(z)$, as $n\to\infty$, the sequence of balls $\{U_n\}_{n\in\N}$ given by
$U_n:=\{X_0>u_n\}$
is a nested sequence of sets such that
$\bigcap_{n\in\N} U_n=\{z\}.$

As explained in \cite{FF08a} the type of asymptotic distribution obtained depends on the chosen observables. In the simulations, we will cover the observables such that 
\begin{equation}
\label{eq:observable-shape}
\varphi(\cdot)=g(\dist(\cdot,z)),
\end{equation} where $\dist$ denotes a certain metric chosen on $I$ and $g$ is of one of the following three types:
\begin{enumerate}
\item $g_1(y)=-\log(y)$ to study the convergence to the Gumbel law.
\item $g_2(y)= y^{-1/a} \qquad a \in \mathbb{R} \quad a>0$ for the Fr\'echet law.
\item $g_3(y)= C -y^{1/a} \qquad a \in \mathbb{R}\quad a>0 \quad C \in \mathbb{R}$ for the Weibull distribution.
\end{enumerate}

%

The sequences of real numbers $u_n=u_n(\tau)$, $n=1,2,\ldots$, are usually taken to be as one parameter linear families like $u_n= y/a_n +b_n$, where $y\in\R$ and $a_n>0$, for all $n\in\N$. In fact, in the classical theory, one considers the convergence of probabilities of the form 
$
\p(a_n(M_n-b_n)\le y).
$
In this case, the {\em Extremal Types Theorem} says that, whenever the variables $X_i$ are i.i.d, if for some constants $a_n>0$, $b_n$, we have 
\begin{equation}
\p(a_n(M_n-b_n)\le y)\rightarrow G(y),
\end{equation}
where the convergence occurs at continuity points of $G$, and $G$ is non degenerate, then $G$ belongs to one of three extreme values types (see below).
Observe that $\tau$ depends on $y$ through $u_n$ and, in fact, in the i.i.d. case, depending on
the tail of the marginal d.f. $F$, we have that $\tau=\tau(y)$ is
of one of the following three types (for some $\alpha>0$):
\begin{equation}\tau_1(y)=\e^{-y} \text{ for } y \in
\mathbb{R},\quad \tau_2(y)=y^{-\alpha} \text{ for } y>0\quad \text{ and }\quad
\tau_3(y)=(-y)^{\alpha} \text{ for } y\leq0.
\label{eq:tau}
\end{equation}
In  \cite[Theorem~1.6.2]{LLR83}, it were given sufficient and necessary conditions on the tail of the d.f. $F$ in order to obtain the respective domain of attraction for maxima. Besides, in  \cite[Corollary~1.6.3]{LLR83} one can find specific formulas  for the normalising constants $a_n$ and $b_n$ so that the respective extreme limit laws apply. We used these formulas to perform the numerical computations in this paper. 

\begin{remark}
\label{rem:relation-iid-EI}
We emphasise that, for i.i.d. sequences, the limiting distribution type of the partial maxima is completely determined by the tail of the d.f. $F$. For the stationary stochastic processes  considered here, if an EI $\theta>0$ applies, then the same can still be said about the limiting distribution type of the partial maxima: namely, it is completely determined by the tail of the d.f. $F$. This statement follows from the equivalence between (\ref{eq:un}) and (\ref{eq:iid-maxima}) and the definition of the EI. However, we also quote \cite[Corollary~3.7.3]{LLR83} because we will refer to it later: 

If for $X_0, X_1,\ldots$ we have an EI $\theta>0$, then if we considered an i.i.d. sequence $Z_0, Z_1,\ldots$ so that the d.f. of $Z_0$ is $F$, the same as that of $X_0$, and let $\hat M_n=\max\{Z_0,\ldots,Z_{n-1}\}$, then the existence of normalizing sequences $(a_n)_{n\in\N}$ and $(b_n)_{n\in\N}$ for which $\lim_{n\to\infty}\p(a_n(\hat M_n-b_n)\leq y)= G(y)$ implies that $\lim_{n\to\infty}\p(a_n(M_n-b_n)\leq y)=G^\theta(y),$ and the reciprocal is also true. Moreover, since by  \cite[Corollary~1.3.2]{LLR83}  $G^\theta$ is of the same type of $G$, we can actually make a linear adjustment to the normalizing sequences $(a_n)_{n\in\N}$ and $(b_n)_{n\in\N}$ so that the second limit is also $G$.
\end{remark}

\begin{remark}
\label{rem:F-dependence}
From Remark~\ref{rem:relation-iid-EI}, in order to determine the type of extremal distribution $G$ (recall that $G(y)=\e^{\tau(y)}$, where $\tau(y)$ is of one of the three types described in (\ref{eq:tau})) which applies to our stochastic processes $X_0, X_1,\ldots$, one needs to analyse the tail of the d.f. $F$. The choice of the observables in (\ref{eq:observable-shape}) implies that the shape of $g$ determines the type of extremal distribution we get. In particular, for observables of type $g_i$ we get an extremal law of type $\e^{\tau_i}$, for $i=1,2,3$. (See \cite[Remark~1]{FFT10} for more details on this correspondence). While the type of the extremal distribution is essentially determined by the shape of the observable, in the cases when types 2 and 3 apply, \ie the Fr\'echet and Weibull families of distributions, respectively, the exponent $\alpha$ is also influenced by other quantities such as the EI and the local dimension of the stationary (invariant) measure $\mu_\epsilon$ ($\mu$). In particular, when such measure is absolutely continuous with respect to Lebesgue and its Radon-Nikodym derivative has a singularity at $z$, then the order of the singularity also influences the value of $\alpha$.    
\end{remark}

\section{Rotations on the circle}
\label{sec:Rotations}

In this section we  study the rotations of the circle and show that the effect of adding noise is enough to create enough randomness in order to make EVLs appear when they do not hold for the original system. This isthe content of Theorem~\ref{th:EVLs-rotations} below. Since this nice statistical behaviour appears solely on account of the noise, it is not surprising that the numerical simulations show that for very small noise one needs a large amount of data to the detect the EVLs. 

\subsection{Analytical results}\label{subsec:rotations-analytical}

In \cite[Theorem~C]{AFV12},  the following result has been proved  essentially for piecewise expanding maps randomly  perturbed  like above. 

To be more precise, let us suppose that the unperturbed map $f$ is continuous on the circle and that \footnote{The result is even more general and applies to multidimensional maps too, but for our concerns, especially for  rotations, the 1-D case is enough. We will discuss later  about generalisation to piecewise continuous maps}:\\
 (i) the correlation integral Cor$_{\eps}(\phi, \psi, n)$ decays at least as $n^{-2}$;\\
 (ii)  $u_n$ satisfies:  $n\p(X_0>u_n)\rightarrow \tau, \ \mbox{as} \ n\rightarrow \infty$;\\
 (iii)  $U_n=\{X_0>u_n\}$ verifies $\bigcap_{n\in\mathbb{N}}U_n=\{z\}$;\\
 (iv) There exists $\eta>0$ such that $d(f(x), f(y))\le \eta \ d(x,y)$, where $d(\cdot, \cdot)$ denotes some metric on $I$, \\
 then the process $X_0, X_1,\cdots$ satisfies $D_2(u_n)$ and $D'(u_n)$, and this implies that the EVL holds for $M_n$ so that $\bar{H}(\tau)=e^{-\tau}.$
 The proof strongly relies on the decay of correlations for $L^1(\l)$ observables. \\
 
In this section we show that such a result holds also for rotations (irrational or not), perturbed with additive noise. The key observation is that in the proof of \cite[Theorem~C]{AFV12}, the observables  entering the correlation integral are characteristic functions of intervals (see also \cite[Remark~3.1]{AFV12}) and for such observables it is possible to prove an exponential decay of correlations for perturbed rotations by using the Fourier series technique. 
 
In what follows we identify $\mathbb{S}^1=\R/\Z$.

\begin{theorem}
\label{th:EVLs-rotations}
Let $f:\mathbb{S}^1\to\mathbb{S}^1$ be a rotation of angle $\alpha\in\R$, \ie $f(x)=x+\alpha\;\mbox{mod }1$. We perturb $f$ additively so that the random evolution evolution of an initial state $x\in\mathbb{S}^1$ is given by  \eqref{eq:concatenated}, with $\vartheta_\eps$ denoting the uniform distribution on $[-\eps,\eps]$. Let $X_0, X_1,\ldots$ be a stationary stochastic process generated by the random evolution of such $f$, as in \ref{eq:def-rand-stat-stoch-proc-RDS2}. Let $(u_n)_{n\in\N}$ be a sequence such that items (ii) and (iii) above hold. Then, the process $X_0, X_1,\ldots$ satisfies conditions $D_2(u_n)$ and $D'(u_n)$ which implies that there exists an EVL for $M_n$ with EI equal to $1$, \ie $\bar H(\tau)=\e^{-\tau}$.
\end{theorem}

For ease of exposition, we will change slightly the notation for the random perturbation of $f$. We write them in this way:
 \begin{equation}
 f_{\eps\xi}(x)=x+\alpha+\eps\xi\ \mbox{mod}1
 \end{equation}
 where $\xi$ is a random variable uniformly distributed over the interval $[-1,1]$ and therefore of zero mean.\footnote{With respect to the previous notations, we changed $\Omega_{\eps}$ into $[-1,1]$, $\omega=\eps \xi$, with $\xi\in [-1,1]$ and finally $\vartheta_{\eps}$ becomes $d\xi$ over $[-1,1]$. } Let us observe that
\begin{equation}
 f^j_{\eps\bar{\xi}}(x)=x+j\alpha+\eps(\xi_1+\cdots+\xi_j)
 \end{equation}

  In this case, it is straightforward to check (just by using the definition), that the stationary measure coincides in this case with the Lebesgue measure, $\l$,  on $[0,1]$ and it is therefore independent of $\eps$.

 Let us first establish that the correlation of the right functions decays exponentially fast under the random evolution of the system. 
 
 \begin{lemma}
 \label{lem:decay-correlations}
 Under the assumptions of Theorem~\ref{th:EVLs-rotations}, if $\phi=\chi_A$ and $\psi=\chi_B$, where $\chi$ denotes the characteristic function, $B=\cup_{l=1}^\ell B_l$, for some $\ell\in\N$ and $A, B_1, \ldots, B_\ell\subset \mathbb S^1$ are connected intervals, then
 \begin{equation}
 \label{dec}
 C_{j,\eps}:=\left|\int_0^1 {\mathcal U}_{\eps}^j(\psi) \phi dx -\int_0^1\psi dx \int_0^1\phi dx\right|\leq 4e^{-j\eps^2\log(2\pi)},
 \end{equation}
as long as $\eps^2<1-\log 2/\log(2\pi).$
 \end{lemma}
\begin{proof}

 We begin by writing the modulus of the correlation integral, $C_{j,\eps}$, as follows:
\begin{equation*} 
 C_{j,\eps}=\left|\int_0^1dx\frac{1}{2^j}\int_{-1}^1d\xi_1\cdots\int_{-1}^1d\xi_j \psi(f^j_{\eps\bar{\xi}}(x))\phi-\int_0^1\psi dx \int_0^1\phi dx\right|
 \end{equation*}
We express $\phi$ and $\psi$ in terms of their respective Fourier series:
 \begin{equation*}
 \psi(x)=\sum_{k\in\mathbb{Z}}\psi_k\,\e^{2\pi ikx} \ \  \ \  \phi(x)=\sum_{k\in\mathbb{Z}}\phi_k\e^{2\pi ikx}
 \end{equation*}
Note that given $A=[a,b]$ and $B_l=[a_l,b_l]$, $l=1,\ldots,\ell$, we have 
\begin{equation*}
\phi_k=\int_0^1 \chi_{A}\, \e^{-2\pi i kx} dx = \frac{\e^{-2\pi i ka}-\e^{-2\pi i kb}}{2\pi i k},
\end{equation*} and also 
\begin{equation*}
\psi_k=\int_0^1 \chi_{B} \e^{-2\pi i kx} dx = \frac{\sum_{l=1}^{\ell}\e^{-2\pi i ka_l}-\e^{-2\pi i kb_l}}{2\pi i k},
\end{equation*} 
which implies that 
\begin{equation}
\label{eq:norm-coefficients}
|\psi_k|, |\phi_k|\le 1/|k|.
\end{equation}
Plugging the Fourier series of $\phi, \psi$ into $C_{j,\eps}$, we obtain
\begin{equation*}
C_{j,\eps}=\left|\sum_{k\in\mathbb{Z}/\{0\}}\psi_k\phi_{-k}\frac{\e^{2\pi ikj\alpha}}{2^j}\int_{-1}^1d\xi_1 \e^{2\pi i k \eps\xi_1}\cdots \int_{-1}^1d\xi_j \e^{2\pi i k \eps\xi_j}\right|=\left|\sum_{k\in\mathbb{Z}/\{0\}}\psi_k\phi_{-k}S^j(k\eps)\right|
\end{equation*}
where $ S(x)=\frac{\sin(2\pi x)}{2 \pi x}$. For this quantity we have the bounds: $|S(x)|\le \e^{-x^2\log(2\pi)}, \ \ |x|<1$, and $|S(x)|\le \frac{1}{2\pi |x|}, \ \ |x|\ge 1.$\\
We now continue to estimate $C_{j\eps}$ as
\begin{equation}\label{CI}
C_{j\eps}\le2\sum_{k=1}^{1/\eps}|\psi|_k|\phi_{-k}|\e^{-k^2\eps^2j\log(2\pi)}+\sum_{k=1/\eps}^{\infty}|\psi|_k|\phi_{-k}|\frac{1}{(2\pi k\eps)^j}
\end{equation} 
Using \eqref{eq:norm-coefficients}, we have
\begin{equation*}
C_{j\eps}\le 2 \e^{-\eps^2 j \log(2\pi)}\sum_{k=1}^{1/\eps}\frac{1}{k^2}+\frac{2}{(2\pi)^j}\sum_{k=1/\eps}^{\infty}\frac{1}{k^2}
\end{equation*}
Using the inequality $\sum_{k=A}^B\frac{1}{k^2}\le \frac2A-\frac1B$, we finally have
\begin{equation*}
C_{j,\eps}\le2(2-\eps)\e^{-\eps^2j\log(2\pi)}+4\eps \e^{-j\log(2\pi)}\le 4\e^{-j\eps^2\log(2\pi)}
\end{equation*}
which is surely verified when $\e^{-j\eps^2 \log(2\pi)}>2 \e^{-j\log(2\pi)}$, namely $\eps^2<1-\log 2/\log(2\pi)$.
\end{proof}

\begin{proof}[Proof of Theorem~\ref{th:EVLs-rotations}]
We now check the conditions $D_2(u_n)$ and $D'(u_n)$ using the previous estimate for the exponential decay of random correlations for characteristic functions. We will use a few results established on the afore mentioned paper \cite{AFV12}. We begin with $D_2(u_n)$ and we consider, with no loss of generality, the observable $g(x)=-\log(d(x,z))$, where $d(\cdot,\cdot)$ is the distance on the circle and $z$ is a given point on the circle. The event $U_n:=\{g>u_n\}=\{X_0>u_n\}$ will be the ball $B_{\e^{-u_n}}(z).$ We now introduce the observables
\begin{align*}
\phi(x)&=\chi_{\{ X_0>u_n \}}=\chi_{\{g>u_n\}} \label{def:obs_phi},\\
\psi(x)&=\displaystyle\int \chi_{\{ g,\, g\circ f_{\eps\xi_1}, \,\ldots\,,\, g\circ f^{\ell-1}_{\eps \x}\leq u_n\}}\, d\x^{l-1}.
\end{align*}
It has been shown, in \cite[Sect. 4.1]{AFV12}, that proving condition $D_2(u_n)$ can be reduced to estimating the following correlation
\begin{multline*}
\Big|\int\mu_\varepsilon\big(g>u_n,g\circ f^{t}_{\eps\x}\leq u_n,\,\ldots\,, g \circ f^{t+\ell-1}_{\eps\x} \leq u_n \big)\, d\xi^ \N-\\
\int\mu_{\varepsilon}(g>u_n)\, d\xi^ \N\int\mu_{\varepsilon}\big(g\leq u_n, g\circ f_{\eps\xi_1}\leq u_n,\,\ldots\,, g\circ f_{\eps\x}^{\ell-1}\leq u_n \big)\, d\xi^\N \Big|\\
\end{multline*}

Since all the maps $f_{\eps \xi}$ are globally invertible and linear, the events $(g>u_n)$ and $(g\circ f^{t}_{\eps\x}\leq u_n,\,\ldots\,, g \circ f^{t+\ell-1}_{\eps\x} \leq u_n )$ are a connected interval and a finite union of connected intervals, respectively. Therefore we can apply directly formula (\ref{dec}) with an exponential decay in $t$ so that any sequence $t_n=n^{\kappa}$, with $0<\kappa<1$, will allow to verify the condition $n\gamma(n,t_n)\rightarrow 0$, when $n\rightarrow \infty$.\\The computation of $D'(u_n)$ is a bit more lengthy. Let us fix a sequence $\alpha_n\rightarrow \infty$ so that $\alpha_n=o(k_n)$. Next, we introduce the quantity $R^{\x}(A)$ as the {\em first return of the set} $A$ into itself under the realisation $\x$. As in \cite[Sect.~4.1, just before eq. (4.5)]{AFV12}, we have:
$$
n\sum_{j=1}^{\lfloor n/k_n \rfloor}\p(U_n\cap f_{\x}^{-j}(U_n))\leq n\sum_{\alpha_n}^{\lfloor n/k_n \rfloor}\p(\{(x,\x):x\in U_n,\, f_{\x}^j(x)\in U_n,R^{\x}(U_n)>\alpha_n\})$$ 
\begin{equation}\label{pii}
\quad+n\sum_{j=1}^{\lfloor n/k_n \rfloor}\p(\{(x,\x):x\in U_n, R^{\x}(U_n)\leq\alpha_n\})\end{equation} It can be shown that the measure of all the realisation such that $R^{\x}(U_n)\leq\alpha_n$ is bounded by a constant times $\l(U_n)$ times $\alpha_n$ (we use here the fact that $\eta$ in item (iv) above is $1$ in our case). Therefore the second term in (\ref{pii}) is bounded by a constant times:
$$
\frac{n^2}{k_n}\l(U_n)^2\alpha_n.
$$
This reduces to $\tau^2\frac{\alpha_n}{k_n}$, whose limit is 0, as  $n\rightarrow \infty$, given our choices for $\alpha_n$ and $k_n$. We now consider the first term on the right hand side of the inequality (\ref{pii}). We can obtain the following result:
\begin{equation}\label{pip}
n\sum_{\alpha_n}^{\lfloor n/k_n \rfloor}\p(\{(x,\x):x\in U_n,\, f_{\x}^j(x)\in U_n,R^{\x}(U_n)>\alpha_n\})\le 
n\left\lfloor \frac{n}{k_n}\right\rfloor \l(U_n)^2 +n \sum_{j=\alpha_n}^{\left\lfloor \frac{n}{k_n}\right\rfloor}|C_{j\eps}|
\end{equation}
where the correlation integral is computed with respect to the characteristic function of $U_n$. Next we use the decomposition (\ref{CI}) for $C_{j\eps}$; we also observe that the product of the modulus of the $k$-Fourier coefficient of $\chi_{U_n}$ for itself gives
$$
\frac{1}{2\pi^2k^2}(1-\cos2\pi k \l(U_n)).
$$
Using the notation $|U_n|=\l(U_n)$,
we obtain (note that we discard the first term in (\ref{pip}) since it goes to zero because $n^2|U_n|^2\sim \tau^2$):
\begin{multline*}
(\ref{pip})\le n \sum_{j=\alpha_n}^{\left\lfloor \frac{n}{k_n}\right\rfloor}\Big\{2\sum_{k=1}^{1/\eps}\frac{1}{2\pi^2k^2}(1-\cos2\pi k |U_n|)e^{-k^2\eps^2j\log(2\pi)}\\ +2\sum_{1/\eps}^{\infty}\frac{1}{2\pi^2k^2}(1-\cos2\pi k |U_n|)\frac{1}{(2\pi k\eps)^j}\Big\}.
\end{multline*}
We consider the first term into the bracket: an upper bound can be defined as follows:
$$
n |U_n|\sum_{j=\alpha_n}^{\left\lfloor \frac{n}{k_n}\right\rfloor}2e^{-\eps^2j\log(2\pi)}\sum_{k=1}^{1/\eps}\frac{1}{2\pi^2k^2}\left(\frac{1-\cos2\pi k |U_n|}{|U_n|}\right),
$$
When $n\rightarrow \infty$, $n|U_n|\sim\tau$ and the inner series  goes to zero. The second piece in (\ref{pip}) could be  bounded as before by
$$
n |U_n|\sum_{j=\alpha_n}^{\left\lfloor \frac{n}{k_n}\right\rfloor}\frac{2}{(2\pi)^j}\sum_{1/\eps}^{\infty}\frac{1}{2\pi^2k^2}\left(\frac{1-\cos2\pi k |U_n|}{|U_n|}\right)
$$
which converges to zero as well.\end{proof}

\section{Numerical experiments}

For the numerical computations we will use the approach already described in \cite{faranda2011numerical}. It consists in considering unnormalised maxima selected using the block Maxima approach. Once computed the orbit of the dynamical systems, the series of the $g_i$ observables are divided into $m$ blocks of length $n$ observations. Maxima of the observables $g_i$ are selected for each blocks and fitted to a single family of generalised distribution called GEV distribution with d.f.:

\begin{equation}
F_{G}(x; \nu, \sigma,
\kappa)=\exp\left\{-\left[1+{\kappa}\left(\frac{x-\nu}{\sigma}\right)\right]^{-1/{\kappa}}\right\};
\label{cumul}
\end{equation}

which holds for $1+{\kappa}(x-\nu)/\sigma>0 $, using $\nu \in
\mathbb{R}$ (location parameter), $\sigma>0$ (scale parameter) and ${\kappa} \in \mathbb{R}$ is the shape parameter also called the tail index: when ${\kappa} \to 0$, the distribution corresponds to a
Gumbel type ( Type 1 distribution).  When the index is positive, it corresponds to a Fr\'echet (Type 2 distribution); when the index is negative, it corresponds to a Weibull (Type 3 distribution).
%
%
In  this fitting procedure, as we have seen in \cite{faranda2011numerical}, the
following relations hold:
\begin{align}
\text{for type 1,}&\quad \kappa=0, \qquad \nu=b_n, \qquad \sigma=1/a_n;\label{eq:relation1}\\
\text{for type 2,}&\quad \kappa=1/\alpha, \qquad \nu=b_n+1/a_n, \qquad \sigma=\kappa/a_n;\label{eq:relation2}\\
\text{for type 3,}&\quad \kappa=-1/\alpha, \qquad \nu=b_n-1/a_n, \qquad \sigma=-\kappa/a_n.\label{eq:relation3}
\end{align}
where $a_n$ and $b_n$ are the normalising sequences described above.

For the discrete maps like $f_\omega(x)$, we select a value for $z$ and repeat the following operations for various values $\epsilon$ of the noise intensity:
\vskip 10pt
\begin{enumerate}
\item We construct an empirical physical measure of the system by performing a long run and saving the obtained long trajectory. The realisations $f_{\omega_{n-1}},\cdots,f_{\omega_1}$ are constructed by taking the $\omega_{k}$ with the uniform distribution in $\Omega_{\epsilon}$.
\item We select 500 initial conditions according to the physical measure described above; such initial conditions will be used as initial conditions ($x$ variables) to generate 500 realisations of the stochastic process. 
\vskip 5pt
\item For each of the 500 realisations, we obtain an orbit  containing  $r=m\cdot n=10^6, 10^7$ data.
\vskip 5pt
\item We split the series in $m=1000$ bins each containing $n=1000$ (or $n=10000$) observations. These values are chosen in agreement with the investigations previously carried out in \cite{faranda2011numerical}, since they provide enough statistics to observe EVL for maps which satisfy the  $D_2$ and $D'$ conditions.
\vskip 5pt
\item We fit the GEV distribution for the $g_i, i=1,2,3$ observables and the maxima and minima  and study the behaviour of selected parameters.
\end{enumerate}

In the forthcoming discussion we will present the results for the shape parameters $\kappa(g_i)$ and for the location parameter $\sigma(g_1)$. This choice is convenient, as explained in \cite{lucarini2012extreme, faranda2012generalized}, since these parameters do not depend on the number of observations $n$  and therefore their asymptotic expected value does not change when $r$ is modified. For the one dimensional maps considered we expect to obtain the following asymptotic parameters for  systems satisfying the mixing conditions described in the former sections (see  \cite{faranda2012generalized}):

\begin{equation}
\kappa(g_1)=0 \quad \kappa(g_2)=1/\alpha \quad \kappa(g_3)=-1/\alpha \quad \sigma(g_1)=1
\label{lp}
\end{equation}

The inference procedure follows the strategy described in \cite{faranda2011numerical} where the authors have used a Maximum Likelihood Estimation (MLE) that is well defined when the underlying physical measure is absolutely continuous.  In the applications described in \cite{lucarini2012extreme, faranda2012generalized} an inference via a L-moments procedure has been preferred as it provides reliable values for the parameters even when the d.f. corresponds to fractal or multi fractal measures. However, since the L-moments procedure does not give information about the reliability of the fit\footnote{The L-moments inference procedure does not provide any confidence intervals unless these are derived with a bootstrap procedure which is also dependent on the data sample size \cite{burn2003use}. The MLE, on the other side, allows for easily compute the confidence intervals with analytical formulas \cite{royall1986model}.}, hereby  we exploit the MLE procedure since it helps to highlight the situations where the fitting does not succeed because of the poor data sampling.\\

The following numerical analysis is constructed in a strict correspondence with the theoretical setup: the role of the simulations is to follow step-by-step the proofs of the theorems trying to  reproduce the analytical results. Whenever it occurs, the failure of the numerical approach is analysed in relation to the length of the sample and the intensity of the noise, just to highlight which parameters are crucial for the convergence procedure. 

We  focus on simple discrete maps as they already contain interesting features regarding the convergence issues. We start by analysing the rotations on the 1-D torus where the condition the mixing condition $D_2(u_n)$ is not satisfied, but it can be restored under perturbation. In the case of   expanding maps of the interval, for which the existence of  EVL has been already shown both analytically and numerically, we show that the computation of the EI at periodic points reveals to be  crucial to discriminate whether the system was randomly perturbed or not.  Eventually, the analysis of the H\'enon map, shows how the presence of another  basin of attraction (infinity in this case) may affect the EVL regardless the initial condition chosen.

Besides the convergence issues, one of the main problems that we face regarding the estimation of the EI during the realisation of the numerical simulations  is the fact that the fitting procedure used in  \cite{faranda2011numerical}, ``hides'' the EI. This is the case because, as we mentioned in Remark~\ref{rem:relation-iid-EI}, when quoting  \cite[Corollary~3.7.3]{LLR83}, the normalizing constants $a_n$ and $b_n$ can be chosen (after a simple linear rescaling) so that we get exactly the same limit law $G$ of the corresponding i.i.d. process instead of $G^\theta$. This prevents the detection of an EI $0<\theta<1$. Basically, our original procedure simply selects the best fitting constants from the sample of maxima available, which means that the EI does not surface at all. In order to overcome this difficulty, instead of applying a blind fitting, we use relations (\ref{eq:relation1})--(\ref{eq:relation3}) (more specifically (\ref{eq:relation1})), the knowledge about the stationary measure of the particular systems we considered and the information provided by \cite{FFT10, FFT12, AFV12} to compute a priori the normalizing sequences so that we can capture the value of the EI. This adjustment to the procedure has turned out to be  very effective and the results met completely the theoretical predictions.

\subsection{Rotations}
In this section we describe the results obtained from the numerical investigation of the stochastically perturbed rotation map following the procedure described in Section~\ref{sec:Rotations}.  The results are displayed in Fig.\ref{rot} where the  green lines correspond to experiments where we have chosen a bin length $n=10^4$ , whereas the blue lines refer to experiments where the chosen bin length is $n=10^3$. The red lines indicate the values of the parameters predicted by the theory for a 1-D map satisfying the mixing conditions described above. We set $\epsilon=10^{-p}$ to analyse the role of the perturbations on scales ranging from values smaller than those typical for the numerical noise to $\mathcal{O}(1)$ \cite{faranda2012analysis}.\\
The solid lines display the values obtained by averaging over the 500 realisations of the stochastic process, while the error bars indicate the standard deviation of the sample. Finally, with the dotted lines we indicated the  experiments where less than 70\% of the 500 realisations  produce a statistics of extremes such that the empirical d.f. passes successfully  the non-parametric Kolmogorov-Smirnov test  \cite{lilliefors} when compared to the best GEV fit.\\

Even though the theory guarantees  the existence of EVL for the rotations perturbed with an arbitrarily weak noise (e.g. when compared with the numerical noise), the simulations clearly show that EVLs are obtained when considering small but finite noise amplitudes only when very long trajectories are considered. The quality of the fit improves when larger bins are considered (compare blue and green lines in Fig.\ref{rot}). This is in agreement with the idea that  we should get EVL for infinitely small noises in the limit of infinitely long samples. In our case, EVLs are obtained only for $\epsilon>10^{-4}$, which is still considerably larger than the noise introduced by round-off resulting from double precision, as the round-off procedure is equivalent to the addition to the exact map of a random noise of order $10^{-7}$ \cite{knuth1973art,faranda2012analysis}.\\
This suggests that is relatively hard to get rid of the properties of the underlying deterministic dynamics just by adding some noise of unspecified strength and considering generically long time series: the emergence of the smoothing due to the stochastic perturbations is indeed non-trivial when considering very local properties of the invariant measure as we do here. It is interesting to expand the numerical investigation discussed here in order to find empirical laws connecting the strength of the noise perturbations with the length of the time series needed to observe EVLs.

\subsection{Expanding Maps}
\label{bersec}
In \cite{AFV12} it is proved that expanding maps on the circle admit EVLs when they are perturbed with additive noise (see, in the just quoted reference,  Corollary 4.1 for smooth maps and Proposition 4.2 for discontinuous maps on $I$). The proof is completely different of the one we produced for rotations in Theorem~\ref{th:EVLs-rotations}. The proof there relies on the fact that expanding maps have exponential decay of correlations for BV-functions  against $L^1$ observables. Moreover, as we already anticipated,  the proof is not limited to continuous expanding maps, but it also holds for piecewise expanding maps with a finite number of discontinuities, provided that the map is topologically mixing. It is interesting to point out that under the assumptions (items (i)-(iv)) in the beginning of Subsection~\ref{subsec:rotations-analytical}, we could prove convergence to the $e^{-\tau}$ law with the observables $g_i(d(x,z))$ previously introduced independently of the choice of the point $z$.  We should remind that for deterministic systems and, whenever $z$ is a periodic point of prime period $p$,  the limit law is not $e^{-\tau}$ any more, but rather $e^{-\theta\tau}$, where $0<\theta<1$ is  the {\em extremal index} discussed above. Therefore, obtaining experimentally an extremal index with unitary value for an observable $\varphi(\cdot)$ computed with respect to a periodic point $z$ definitely points to the fact that the system is stochastically perturbed. This should be also  true for general dynamical systems (including non-uniformly
hyperbolic systems) as it concerns general properties of the noise and its ability of destroying cluster dynamics. The last statement implies both theoretical and numerical verifications that will reported elsewhere.

If a non-periodic point $z$ is considered, it seems relevant to investigate whether stochastic and deterministic systems exhibit any sort of differences when looking in detail into the resulting EVLs.
We consider the following map:
 \begin{equation}
 f_{\eps\xi}(x)=3x+ \eps\xi\ \mbox{mod}1
 \end{equation}

perturbed with additive noise as in the rotation case discussed above. The stationary  measure for such a map is the uniform Lebesgue measure on the unit interval independently of the value of $\eps$. The numerical setup is  the same as in the case of the rotations: we begin by taking a non-periodic point as the center of our target set; the results are shown in picture \ref{ber}. It is clear that the stochastic perturbations do not introduce any changes in the type of statistical behaviour observed for a non-periodic point $z$ and no differences are encountered even if the number of observations in each bin is increased. This is compatible with the idea that the intrinsic chaoticity of the map overcomes the effect of the stochastic perturbations. Summarising, extremes do not help us in this case to distinguish the effect of intrinsic chaos and the effect of adding noise.\\

It seems more promising the problem of the form of the EVLs for periodic points. As discussed above, we cannot use  the usual fitting procedure for the GEV, since it always renormalises in such a way that the extremal index seems to be one (see discussion above). Instead, in order to observe extremal indices different from one, we have to fit the series of minimum distances to the exponential distribution by normalising a priori the data. The inference of the extremal index $\theta$ (here already  the extremal index) has been obtained via MLE.  The normalisation applied consists only in multiplying the distances by the factor $2n$ as we deal with a constant density measure. In Fig. \ref{ei} we present the results for the extremal index obtained taking the periodic point $z=1/2$ of prime period 1 for which $\theta=2/3$ \footnote{We recall that  $\theta=\theta(z)= 1 - |\mbox{det} D(f^{-p}(z)|$, where $z$ is a periodic point of prime period $p$, see \cite[Theorem 3]{FFT12}.}. The results  clearly show that we are able to recognise the perturbed dynamics as the extremal index goes to one when $\epsilon$ increases.  Interestingly, the rupture with the values expected in the purely deterministic case is observed only for relatively intense noise. Moreover, when longer time series are considered (green experiment), the stochastic nature of the map becomes evident also for weaker perturbations. Finally, it is clear that the numerical noise (corresponding to $\epsilon \simeq 10^{-7}$) is definitely not sufficiently strong for biasing the statistics of the deterministic system.

\subsection{Quadratic map}

Following the analysis on the extremal index carried out in the previous section, we study numerically the quadratic map 
$$ f_{\eps\xi}(x)=1-ax^2+ \eps\xi $$
in the two cases $a=0.014$ and $a=0.314$. For these values of $a$, the deterministic dynamics is governed by the existence of an attracting fixed point (whose value depends on $a$); for both choices of $a$ the fixed point is taken as $z$ in the experiments. The stationary  measure of such perturbed quadratic maps is absolutely continuous with respect to Lebesgue \cite{Baladi96}. Extremes of this map have been already studied in the unperturbed case in \cite{FF08}, but for chaotic regime ($a$ belonged to the Benedicks-Carleson set of parameters). Here, we proceed exactly as described in the previous section, by fitting the series of minimum distances to an exponential distribution after normalizing the data using the factor $2n$. From the theoretical point of view, the choice of this normalisation comes from assuming that, in the presence of strong noise, the measure becomes smooth in the neighbourhood of the periodic point and therefore it is locally not different with respect to the ternary shift already analysed. We believe that  the EI computed with respect to this normalising constant seems to be the relevant one for the simulations in perturbed systems  since it does not depend on the intensity of the noise but only on parameters that are known from the set up of the experiment.

We did several experiments spanning the noise range described in the previous experiments. We found that, up to the deterministic limit the role of noise is such that we can still obtain $\theta=1$ if we renormalise the series of the minimum distances by using the factor $2n/\eps$ instead of just $2n$. In the deterministic limit we obtain in both cases a result for the EI that depends on the length of the bin. This is probably due to numerical uncertainties intrinsic to operating in double precision.  This effect can be clearly recognised by  repeating the same experiment in single precision - results not shown here - where $\theta$ saturates for values of $\epsilon < 10^{-4}$ because the effect of the single precision numerical noise becomes dominant.

\subsection{Maps with indifferent fixed points}

We consider now maps with slow (polynomial) rates of decay of correlations due to the presence of an indifferent fixed point around which orbits linger for a long time. Examples of maps with such behaviour are the ones usually referred to as Manneville-Pomeau maps introduced in \cite{PM80}, which are given by
$$f(x)=x+x^{1+\alpha} + \eps\xi \mod 1$$ and the ones studied by Liverani-Saussol-Vaienti in \cite{LSV99}, which are given by
$$f(x)= \left\{ \begin{array}{lc}
x(1+2^\alpha x^{\alpha}) + \eps\xi &\mbox{ for }x\in [0,1/2),\\
2x-1 + + \eps\xi &\mbox{ for }x\in [1/2,1].
\end{array} \right.
$$
Under some general conditions we know that these maps
have a unique \emph{physical} measure $\mu_\alpha$, where ``physical'' is used in the sense that for a positive Lebesgue measure set of points, the time average of any continuous observable function computed along its orbits converges to the respective spacial average. It is also well known that when 
 $\alpha \ge 1$, the contact at the fixed point is so strong that $\mu_\alpha$ is the Dirac measure at $0$ while for
$\alpha\in (0,1)$, we have that $\mu_\alpha$ is absolutely continuous with respect to the Lebesgue measure. 
Moreover, for $\alpha\in (0,1)$ 
polynomial decay of correlations has been proved e.g. in 
\cite{Y99}, \cite{LSV99}. In fact, such polynomial upper bounds for the decay of correlations have actually been proved to be sharp. See for example \cite{G04} and \cite{H04}, for lower bounds for the rate of decay of correlations. 

Although these systems present very slow rates of mixing,  one can easily choose a base of induction (like $[1/2,1]$ for the Liverani-Saussol-Vaienti maps, for example) and observe that the first return time induced map  is piecewise expanding. This means that, in our case of interest, the extremal behaviour of the induced map is well established. Moreover, using \cite[Theorem~2.1]{BST03} one can pass the information from the induced system to the original one to conclude that for generic points $z$ we will always observe an EI equal to 1, for all $\alpha\in(0,1)$. In \cite[Theorem~5]{FFT13}, the authors prove in particular a counterpart regarding periodic points of \cite[Theorem~2.1]{BST03}. Hence, we know that for all such maps, at periodic points (except 0), we have an EI less than 1 and given by the derivative at that point and that, at almost every point, there is an EI equal to 1. 
In an ongoing work, the second and the fifth named authors, together with A.C.M. Freitas and M. Todd, using a recent result \cite{H12}, prove that actually, also for these slowly mixing maps, it is still possible to prove a dichotomy which basically says that except at periodic points, the EI is always equal to 1. This includes a deep analysis of the extremal behaviour at the indifferent fixed point.

We highlight the following interesting aspect about these maps with indifferent fixed points.
In order to prove conditions $D_2(u_n)$ or $D^p(u_n)$, we need that the rate function $\gamma(n,t)$ to go to 0 faster than $t^{-1}$. As it can be seen in \cite[Section~5.1]{F12}, we have always used in all examples rates of decay of correlations faster than $t^{-1}$. However, for parameters $\alpha\in(1/2,1)$, as it has been proved in \cite{G04}, for example, the rate of decay of correlations is, instead, not faster than $t^{1-1/\alpha}$
. This suggests that conditions $D_2(u_n)$ or $D^p(u_n)$ cannot be proved, at least from decay of correlations as usual. Nonetheless, using \cite[Theorem~2.1]{BST03} and \cite[Theorem~5]{FFT13} we can still get EVLs for the original systems and for those problematic slow mixing parameters $\alpha\in(1/2,1)$.

The slow convergence to the EVLs for $\alpha\in(1/2,1)$ can be highlighted by the numerical experiments we have performed on  the Pomeau-Manneville map perturbed by additive noise.  In fact, even if the theory states that EVLs exist for the deterministic dynamics when correlations decay  polynomially, when  finite samples are taken into account, it becomes computationally extremely expensive to observe convergence to the theoretical expected parameters as the bin length must be sensibly increased in order to produce sufficiently independent maxima \footnote{see the discussion  and the examples  about the so-called Standard map in \cite{faranda2012generalized}.}. On the contrary, here we show how  the use of additive noise allows for deducing  interesting properties of the dynamics without employing too many computational resources.  

The experiment follows the set-up already described in the former section: 500 realizations of the perturbed mapat $m=1000$ and $n=1000$  have been performed for decreasing noise intensities $p$ and for ten different values of $a(0,1)$. The points $z$ have been chosen on the physical measure obtained by iterating the map starting from random initial conditions. Only the parameters $\kappa$ and $\sigma$ for the observables $g_1$ are represented in Fig. \ref{PM} although the results obtained are analogous for the $g_2$ and $g_3$ parameters (non shown here).  For $a \in (0, 1/2)$  we obtain similar results both in the deterministic regime and in the stochastic one:   the shape parameter is close to zero   and the scale parameter is close to 1 as expected from the theory.  Instead, when $a(1/2,1)$ we still obtain the results just stated for large noise intensities but, in the deterministic limit, divergences appear as an effect of the slow mixing. It is evident that for these range of parameters the stochastic analysis reveals aspects of the dynamics which would have been precluded by simply analysing the deterministic map. In fact, if in the deterministic case ($p>7$), we are unable to fit a GEV distribution already for $a>0.5$, the addition of the noise suggests that the mixing gets slower  and slower as $a \to 1$. In order to explain this, let  us analyse in detail what happens in the cases $a=0.6$ and $a=0.9$. In the former case, we still obtain the results expected for mixing dynamics up to noise intensity of order of $10^{-6}$, whereas in the case $a=0.9$ the signature of the divergence is already present for noise intensity of order $10^{-4}$.

In summary, $\alpha \to 0$ and $\alpha \to 1$ correspond to two very different regimes, the former case behaving exactly like the ternary shift described in section \ref{bersec}, the latter approaching the behaviour of a map where the dynamic is concentrated around a fixed point, as in the quadratic map examples. 

\subsection{Lorenz Maps}

An interesting question to address when considering stochastic perturbations of deterministic maps concerns the cases when complex bifurcation structures affects the dynamics  depending on a certain parameter $a$.  A good test platform in this sense is represented by the so called Lorenz  maps family. A Lorenz map is constructed by projecting the dynamics arising from the well known Lorenz attractor \cite{lorenz1963deterministic} on the two dimensional manifold obtained without taking into account the most stable direction \cite{fiedler2001ergodic}.

Leaving aside the numerical complications deriving by the direct analysis of the Lorenz equations, we stick to the analysis of one dimensional maps  useful to catch most of the features of the Lorenz equations bifurcation diagram. There are several examples of one dimensional Lorenz maps and they are usually described in a piecewise fashion. We focus here our attention on the interval map with two monotonous continuous branches and a discontinuity in between

$$f(x)=(-a+ |x|^a)\cdot \mathrm{sgn}(x) +\eps\xi, $$

which has been widely studied in \cite{bruin1996wild}. For $\eps=0$ the discontinuity is  at 0 and it is of order $a>0 : f(\pm \delta) \sim f(0)\pm \delta^\alpha$. In the classical Lorenz system  we obtain a structure similar to the one produced by this map if $a<1$ such that  the derivative of the map is infinite at the discontinuity  \cite{fiedler2001ergodic}. 

Even if the map show a very  complicated structure of the bifurcation diagrams  - upper panel of Fig \ref{lor} for the parameter $a \in (0, 2)$ - we are able to classify which kind of EVLs hold for most of the values of $a$ by leading back to the examples already discussed:

- For   $a$  such that the deterministic dynamics is governed by the existence of fixed points or periodic trajectories, the results on  EVLs described for the quadratic map holds: for increasing noise we get convergence to the three classical EVLs  whereas in the deterministic case the EVL obviously tends to a Dirac delta.

- For $a$ such that the dynamic is mixing and supported by an absolutely continuous physical measure, we found exactly the same behaviour described in the case of the ternary shifts.

Several numerical experiments performed in various region of the bifurcation diagram - non shown here - have confirmed these claims. 
Hereby we will instead discuss the cases   $a\to 1$ as they  correspond to the change in the nature of the discontinuity at 0 from infinite ($a<1$) to finite ($a>1$). The numerical experiments presented in the middle and lower panel of Fig.  \ref{lor} for the parameters $\kappa(g_1)$ and $\sigma(g_1)$, $m=n=1000$,  inspect the dynamics for five different values of $a$ sampled in the grey region on the bifurcation diagram. The experimental set-up follows the one described for the Pomeau-Manneville case with 500 realizations of the map and points $z$ on the physical measure. 

In the deterministic limit  all the cases analysed follow  the theoretical expectations: the EVLs diverges from the three classical families and the distributions of $g_1$ maxima approach a Dirac delta  distribution as   $\kappa(g_1) \to 1$ and $\sigma(g_1)\to 0$.  In principle, with the addition of noise,   $\kappa(g_1)$ and $\sigma(g_1)$ should respectively converge to $0$ and $1$ - red solid line in the figures. Instead we get  a divergence from the theoretical values  whose magnitude  increases as $a$ increases. For the values of $a\to 1$ the dynamics under the effect of the noise is governed by jumps  among the stable fixed points represented in the area coloured in grey in the bifurcation diagram of Fig. \ref{lor}. These jumps occur in a rather uncontrolled way as  $p<6$ for $a=0.98$ and $p<5$ for $a=0.99$ . This produces the observed divergence and suggests that the addition of the noise must be carefully  implemented.  When the structure of the phase space involves different basins of attraction, noise may cause uncontrolled jumps which, for finite samples, produce divergence from the classical EVLs meaning that geometrical properties of the measure may be wrongly detected as in the case $a=0.99$. Of course the divergence can be reduced by increasing $m$  since by extending the bin length we would be able to collect more statistics around the point $z$ but this may be beyond our technical computational limits. On the other hand, even in this case the addition of noise can provide interesting hints on the dynamics:  for example this kind of divergence can be  used for detecting the proximity of bifurcation points as discussed in in \cite{faranda2012extreme}.

\subsection{Anosov Diffeomorphisms}

An Anosov map on a manifold $M$ is a certain type of mapping, from $M$ to itself, with rather clearly marked local directions of 'expansion' and 'contraction'\cite{franks1971anosov}.  The problem of classifying manifolds that admit Anosov diffeomorphisms has always been  very difficult, and still nowadays has no answer.  Very few classes of Anosov diffeomorphisms are known. The best known example is the so called Arnold's cat map  introduced in  \cite{arnold1968ergodic}:

$$f(x,y)=(2x+y+\eps\xi\,x+y+\eps\xi ) \mod 1$$

The deterministic map acts in such a way that first it stretches a unit square and then use the modulus operation in order to  rearrange the pieces back into the square. It has  a unique hyperbolic fixed point and its eigenvalues are one greater and the other smaller than 1 in absolute value, associated respectively to an expanding and a contracting eigenspace which are also the stable and unstable manifolds. The eigenspace are orthogonal because the matrix is symmetric.  Both the eigenspaces densely cover the torus that represents the physical measure of the map.

It is interesting to perturb this map stochastically and use the extreme value theory described so far to understand if the noise changes something with respect to the properties observed in the deterministic cases. In \cite{faranda2011numerical} the deterministic map has been already analysed by finding the EVLs expected for maps which satisfy conditions $D_2$ and $D'$. The convergence have been observed for the parameters:
\begin{equation}
\kappa(g_1)=0 \quad \kappa(g_2)=1/(2\alpha) \quad \kappa(g_3)=-1/(2\alpha) \quad \sigma(g_1)=1/2
\label{lp}
\end{equation}

where the factor 2 takes care of the dependence on the dimension of the phase space. The set-up is exactly the same described as in Section \ref{bersec} and the results are shown in Fig. \ref{cat}.  As in the case of the ternary shift, nothing changes with the addition of the noise. It is evident that the intrinsic chaoticity of the deterministic dynamics overcome the effect of the noise. 

The experiments provided that the Arnold cat map does not allow for understanding the role of contracting directions as there is no 'trace'  of the stable manifold on the physical measure since it  fills all the bidimensional  torus. In the next section we analyse the H\'enon map where instead the presence of contracting eigenspace  acts on the physical measure which is constrained on a strange attractor.

\subsection{H\'enon attractor}
Finally, we investigate the impact of a stochastic perturbation applied onto a prototypical case of a map possessing a singular physical measure supported on an attractor  whose dimension is smaller than the dimension of the phase space, $d$. For such dynamical systems in \cite{lucarini2012extreme}  we have shown numerically the convergence of maxima for the observables $g_i$ to the three classical EVLs whose parameters ($\kappa, \sigma, \mu$) depend on the scaling exponent of the measure of a ball centred around the point $z$ from which the distance is computed. Such a scaling exponent turned out to be the local dimension $d_L(z)$ of the SRB-measure. For this reason, the equations for the asymptotic parameters presented in Eq. \ref{lp} for the one dimensional and absolutely continuous case,  should be modified as follows:
\begin{equation}
\kappa(g_1)=0 \quad \kappa(g_2)=1/(ad_L(z)) \quad \kappa(g_3)=-1/(ad_L(z)) \quad \sigma(g_1)=1/d_L(z).
\label{Hp}
\end{equation} 
In \cite{lucarini2012extreme} we have also shown that the dimension of the whole attractor can be recovered by averaging over different points $z$ the local dimensions $d_L(z)$ obtained inverting the relations in Eq. \eqref{Hp}.

Hereby, we choose to investigate the properties of the H\'enon map defined as:
\begin{equation}
\begin{array}{lcl}
x_{j+1}&=&y_j +1 -a x_j^2 + \epsilon \xi_j \\
y_{j+1}&=&b x_j,\\
\end{array}
\end{equation}
where we consider the classical parameters $a=1.4$, $b=0.3$ \footnote{We remind that Benedicks and Carleson \cite{BC91} proved that there
exists a set of positive Lebesgue measure $S$ in the parameter
space such that the H\'enon map has a strange attractor whenever $a,b\in S$. The value of $b$ is very small and the attractor
lives in a small neighbourhood of the $x$-axis. For those values
of $a$ and $b$, one can prove the existence of the physical measure and of a stationary measure under additive noise, which
is supported in the basin of attraction and that converges to
the physical measure in the zero noise limit \cite{BV06}. 
It is still unknown whether such results could be extended to the ''historical'' values that we consider here.}. Recall that for the usual H\'enon attractor, points outside the basin of attraction escape  to infinity. 
Hence, when the system is perturbed,  it is natural to expect that, after waiting enough time, the evolution of every initial condition will escape to infinity since the noise lets the orbit explore the whole $d$-dimensional phase space. Moreover, as we have already pointed out in the introduction, since the physical measure is indeed smoothed by the action of the noise, so that the resulting attractor has the same dimension  $d$ of the phase space. Therefore, the EVLs'  parameters depending on the local dimension $d_L(z)$ of the attractor should change value abruptly as soon as the noise is switched on.
We have selected 500 different $z$ on the support of the physical measure,  and a single realisation has been analysed for each of these $z$. For the  deterministic dynamics, such a set up has been used to reconstruct the attractor's dimension by averaging the local dimension $d_L(z)$ derived for each considered $z$ from the EVLs \cite{lucarini2012extreme}.  Although there are no rigorous results on this system, we will such the numerical evidences and issues emerging from the analysis of H\'enon to suggest how to operate in a general case where theoretical results are - for some reasons - unavailable.\\

In the case of purely deterministic dynamics, the EVLs parameters agree with high precision with the theoretical predictions of \cite{lucarini2012extreme}. The numerical experiments performed with varying intensity of the noise forcing (see Fig. \ref{henon}) suggest that, also in this case, the signature of the deterministic dynamics is pretty resilient. We need a rather intense noise to obtain a detectable smearing of the measure, such that our indicators see an absolutely continuous with respect to Lebesgue, invariant measure.  Only when $p\simeq 0.1$ we observe the escaping behaviour to infinity which causes a divergence of the EVL from the deterministic one with unreliable parameters, whereas in all the other cases we recover quite well the dependence on the dimension of the usual H\'enon attractor. Again, when dealing with finite data samples the correct reconstruction of the phase space and of the physical measure depends on the intensity of the noise. The case of H\'enon also calls the attention for the fact that the balls of the perturbed systems keep scaling with the local dimension when the system is weakly perturbed, instead of scaling with the dimension of the phase space. This effect is explained in the paper by Collet \cite{collet1988stochastic}  and is linked to the direction in which the stochastic perturbation operates: we have already seen for the shift map that  the addition of noise in chaotic one dimensional maps does not affect much the system's behaviour at generic points.
The effect of noise is negligible for the components parallel to the unstable manifold, while the effect is definitely stronger along the stable manifold. Introducing a noise acting as a forcing only along the stable manifold would indeed create a stronger smoothing effect. We will try to investigate this possibility in a future publication. The heuristic take-home message we learn from the case of the H\'enon map is that the noise is not an easy fix for the singularity of the invariant measure in the case of dissipative systems, in the sense that its effect becomes noticeable only when its intensity is relatively larger, unless we consider extremely long time series. 

As discussed in the introduction, this is a delicate point related to the practical applicability of the FDT in dissipative statistical mechanical systems. While it is clear that FDT does not apply for systems possessing a singular invariant measure \cite{ ruelle1998general, ruelle2009review, lucarini2008response, lucarini2011a}, some authors \cite{abramov2007blended, lacorata2007fluctuation}  have advocated the practical applicability of the FDT thanks to the smoothing effect of noise. It is clear that, while even an arbitrarily small noise indeed smooths the invariant measure, one may need to accumulate an extraordinarily long statistics record in order to observe it, and, therefore, to be able to apply for real the FDT. Obviously, this is a matter worth investigating on its own. 

\begin{figure}[H] \begin{center}
\advance\leftskip-1.5cm
       \includegraphics[width=0.8\textwidth]{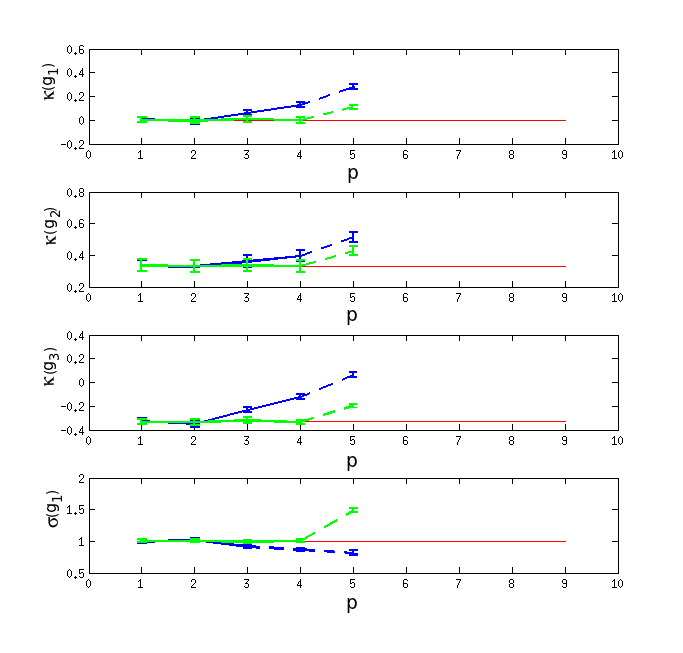}
        \caption{GEV parameters VS intensity of the noise $\epsilon=10^{-p}$ for the circle rotations perturbed map. Blue: $n=10^3$, $m=10^3$, Green: $n=10^4$, $m=10^3$. Red lines: expected values. $ z \simeq 0.7371$. From the top to the  bottom: $\kappa(g_1), \kappa(g_2), \kappa(g_3), \sigma(g_1)$.}
      \label{rot}
    \end{center} \end{figure}

\begin{figure}[H] \begin{center}
\advance\leftskip-1.5cm
       \includegraphics[width=0.8\textwidth]{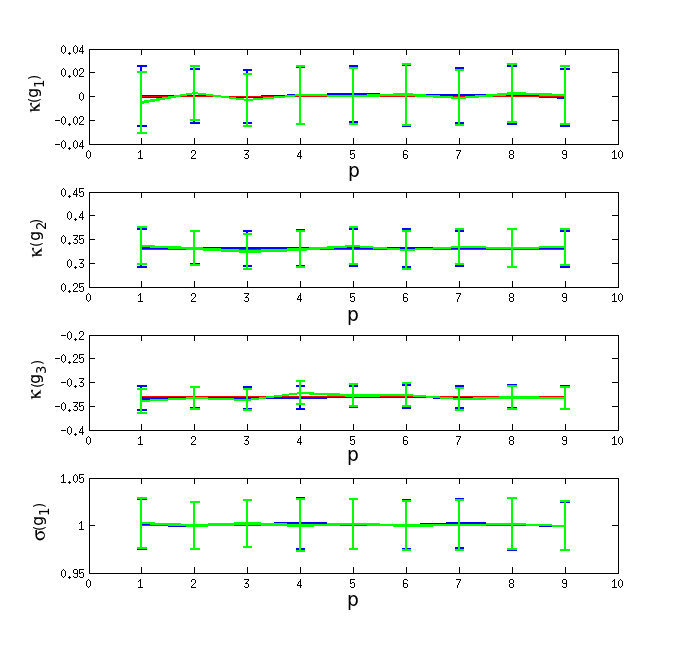}
        \caption{GEV parameters VS intensity of the noise $\epsilon=10^{-p}$ for the ternary shift perturbed map. Blue: $n=10^3$, $m=10^3$, Green: $n=10^4$, $m=10^3$. Red lines: expected values. $z\simeq 0.7371$. From the top to the bottom: $\kappa(g_1), \kappa(g_2), \kappa(g_3), \sigma(g_1)$.}
      \label{ber}
    \end{center} \end{figure}

\begin{figure}[H] \begin{center}
\advance\leftskip-1.5cm
       \includegraphics[width=0.8\textwidth]{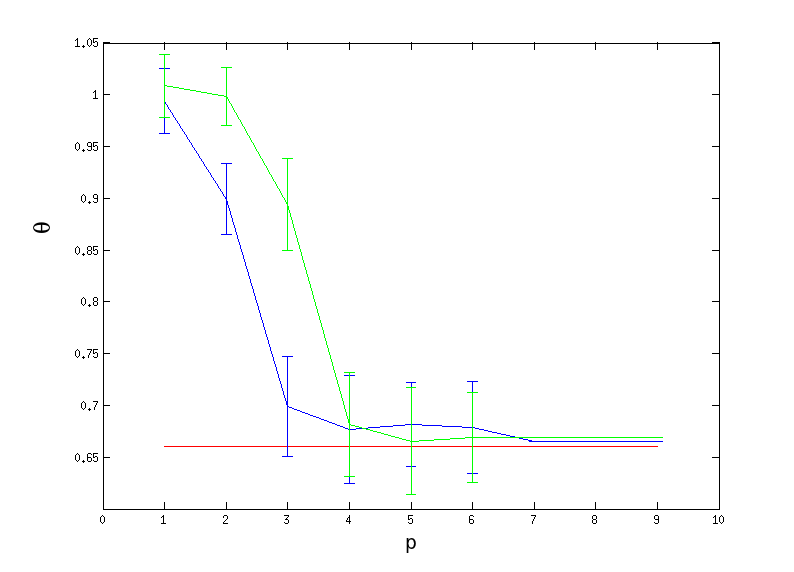}
        \caption{Extremal index $\theta$ VS intensity of the noise $\epsilon=10^{-p}$ for the ternary shift perturbed map. Blue: $n=10^3$, $m=10^3$, Green: $n=10^4$, $m=10^3$. Red line: theoretical $\theta$ for $z=0.5$ of the unperturbed map.}
      \label{ei}
    \end{center} \end{figure}

%

\begin{figure}[H] \begin{center}
\advance\leftskip-1.5cm
        \includegraphics[width=0.8\textwidth]{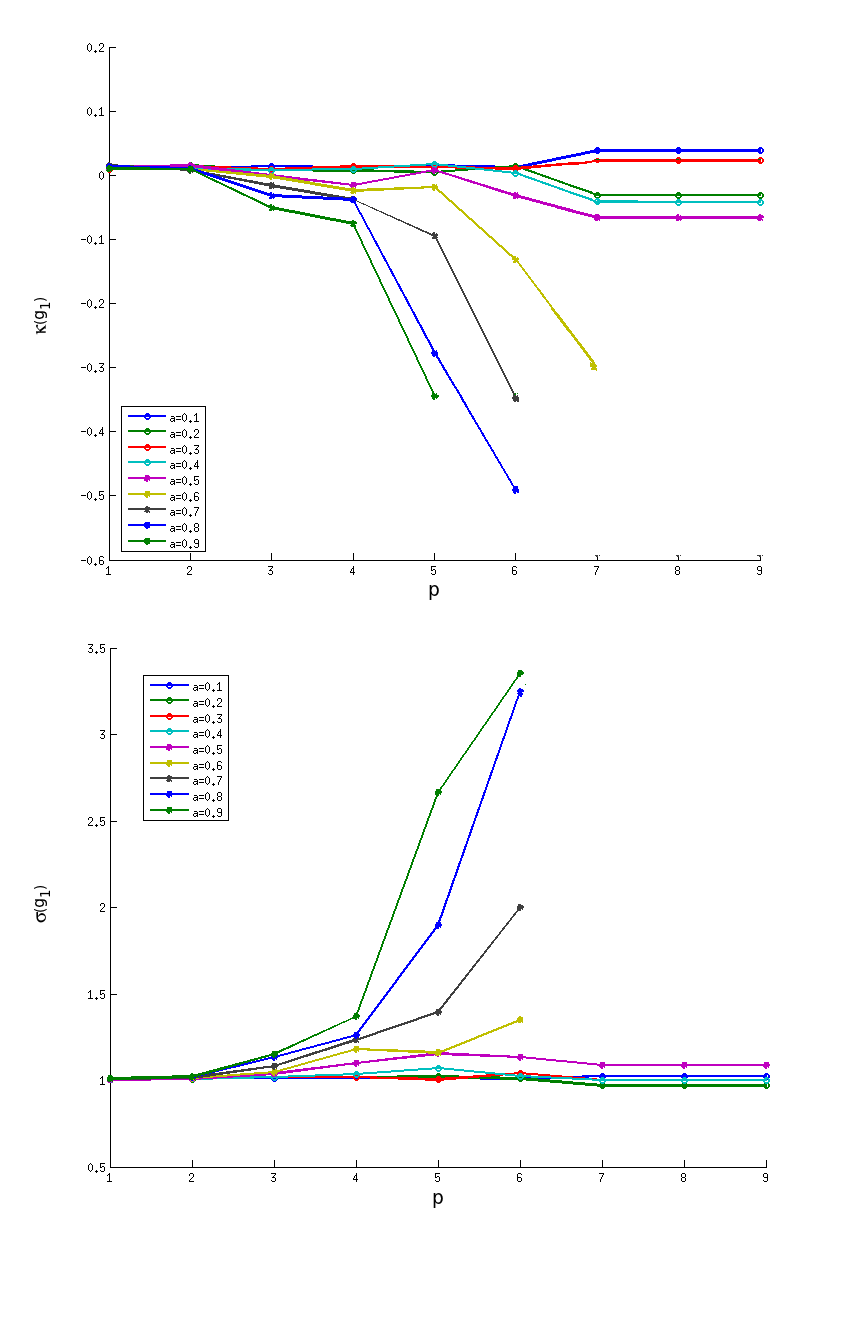}
        \caption{GEV parameters VS intensity of the noise $\epsilon=10^{-p}$ for the Pomeau-Manneville  perturbed map for different values of $\alpha$ (see legend and text for more explanations), $m=n=1000$. From the top to the bottom: $\kappa(g_1), \sigma(g_1)$.}
      \label{PM}
    \end{center} \end{figure}

\begin{figure}[H] \begin{center}
\advance\leftskip-1.5cm
        \includegraphics[width=0.8\textwidth]{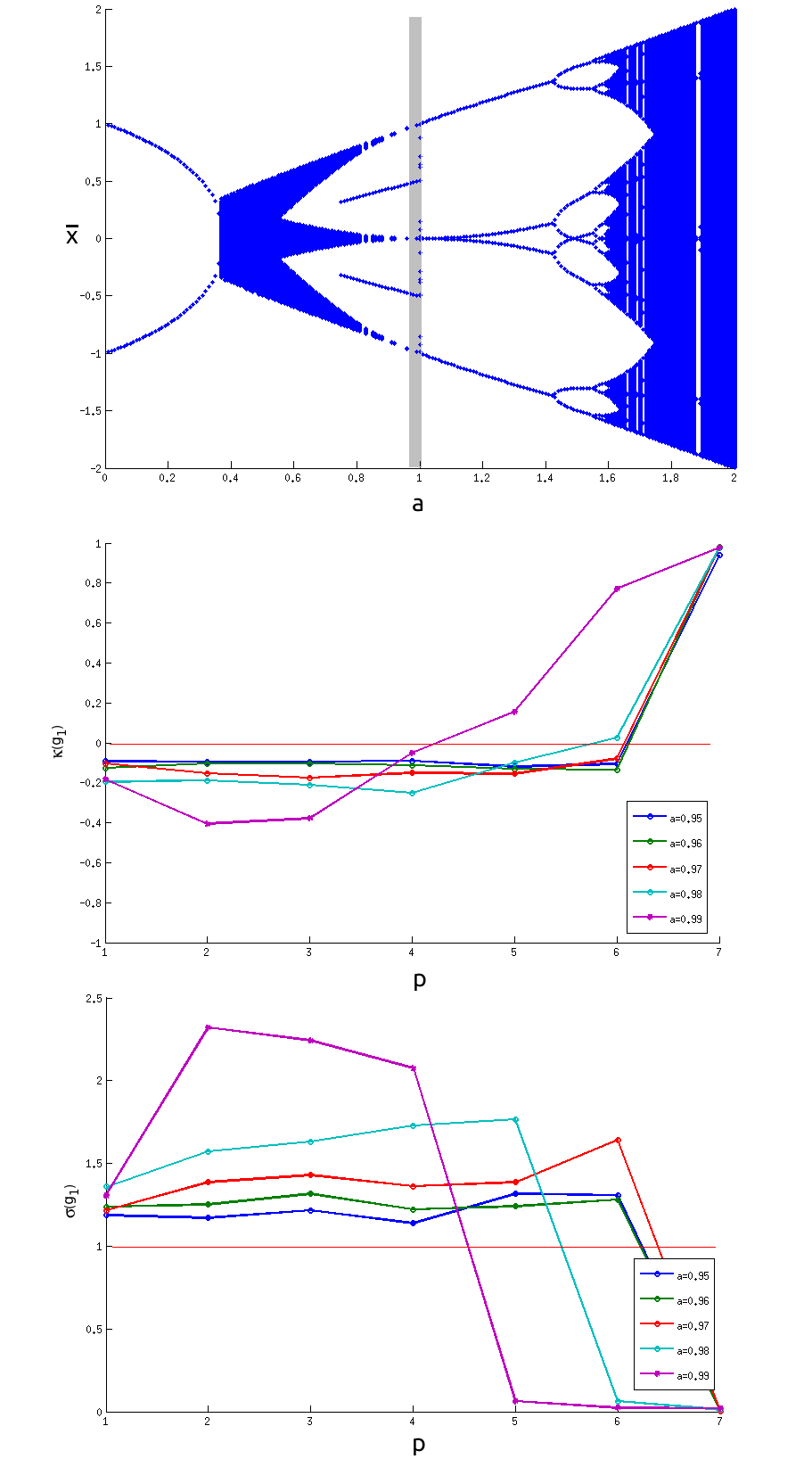}
        \caption{Bifurcation diagram of the Lorenz  deterministic map for $a \in (0 2)$ (upper panel). GEV parameters VS intensity of the noise $\epsilon=10^{-p}$ for the Lorenz map perturbed with additive noise for 5 different values of $a$  chosen in the grey shaded region of the upper plot (see legend and text for more explanations), $m=n=1000$:  $\kappa(g_1)$ (middle panel), $\sigma(g_1)$ (lower panel).}
      \label{lor}
    \end{center} \end{figure}
\begin{figure}[H] \begin{center}
\advance\leftskip-1.5cm
       \includegraphics[width=0.8\textwidth]{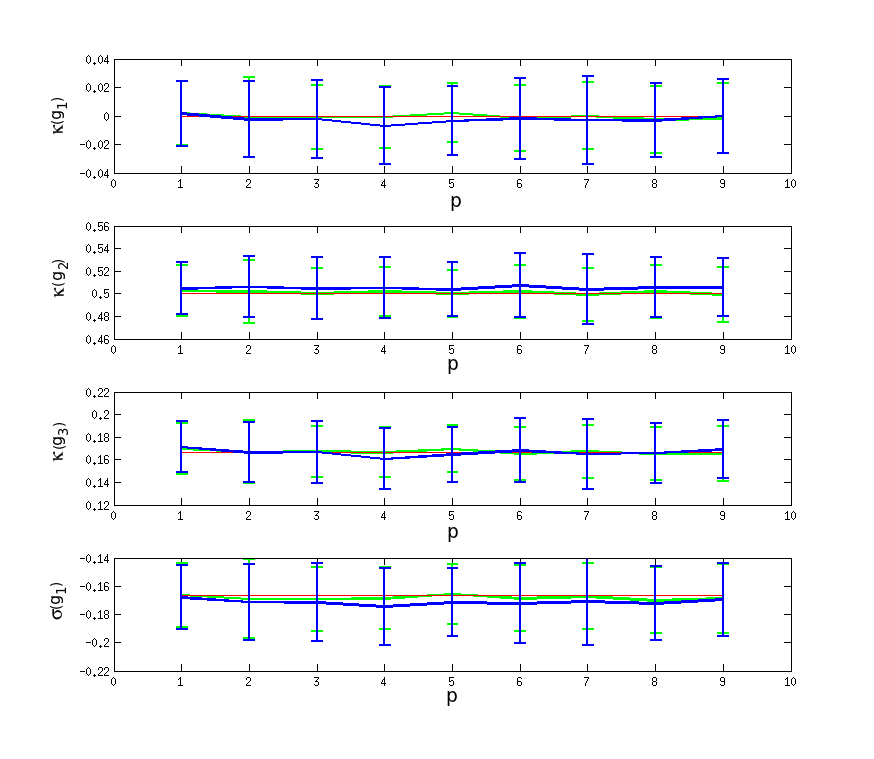}
        \caption{GEV parameters VS intensity of the noise $\epsilon=10^{-p}$ for the Arnold Cat map. Blue: $n=10^3$, $m=10^3$, Green: $n=10^4$, $m=10^3$. Red lines: expected values. From the top to the bottom: $\kappa(g_1), \kappa(g_2), \kappa(g_3), \sigma(g_1)$.}
      \label{cat}
    \end{center} \end{figure}

\begin{figure}[H] \begin{center}
\advance\leftskip-1.5cm
        \includegraphics[width=0.8\textwidth]{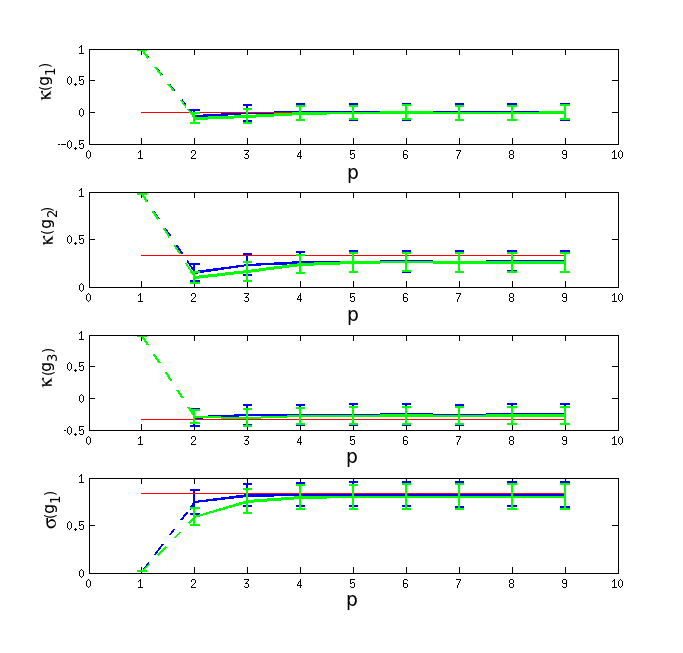}
        \caption{GEV parameters VS intensity of the noise $\epsilon=10^{-p}$ for the H\'enon  perturbed map. Blue: $n=10^3$, $m=10^3$, Green: $n=10^4$, $m=10^3$. Red lines: expected values. $z$ is different  for each realisation. From the top to the bottom: $\kappa(g_1), \kappa(g_2), \kappa(g_3), \sigma(g_1)$.}
      \label{henon}
    \end{center} \end{figure}

\bibliographystyle{abbrv}
\bibliography{NoisyExtremes}

\end{document}